\documentclass[12pt]{article}
\usepackage[margin=2in,includefoot,footskip=12pt,]{geometry}
\usepackage{layout}
\usepackage[utf8]{inputenc}
\usepackage{amsmath}
\usepackage{amsfonts}
\usepackage{amssymb}
\usepackage{amsthm,fullpage}
\usepackage{graphics}
\usepackage{float,graphicx}
\usepackage{tikz}
\usepackage[toc,page]{appendix}
\usepackage{hyperref}
\usepackage{xcolor}
\usepackage[export]{adjustbox} 
\usepackage{caption}
\usepackage{subcaption}

\usepackage{makecell}

\theoremstyle{definition}
\newtheorem{thm}{Theorem}[section]
\newtheorem{lem}[thm]{Lemma}
\newtheorem{cor}[thm]{Corollary}
\newtheorem{defn}[thm]{Definition}
\newtheorem*{ex}{Example}
\newtheorem{prop}[thm]{Proposition}
\newtheorem{rem}[thm]{Remark}

\DeclareMathOperator{\aut}{Aut}

\title{On the construction of cospectral nonisomorphic bipartite graphs\footnote{The contents of this article are included in the last author’s MS thesis \cite{hitesh}.}}
\author{M. Rajesh Kannan\thanks{Department of Mathematics, Indian Institute of Technology Kharagpur, Kharagpur 721 302, India. Email: rajeshkannan@maths.iitkgp.ac.in, rajeshkannan1.m@gmail.com } \and Shivaramakrishna Pragada\thanks{Department of Aerospace Engineering, Indian Institute of Technology Kharagpur, Kharagpur 721 302, India. Email: shivaram@iitkgp.ac.in, shivaramkratos@gmail.com } \and Hitesh Wankhede \thanks{Department of Mathematics,  Indian Institute of Science Education and Research Pune, Pune 411 008, India. Email: hitesh.wankhede@students.iiserpune.ac.in, hiteshwankhede9@gmail.com }
}
\date{\today}
\begin{document}
    \maketitle
    \baselineskip=0.25in
    \begin{abstract}
    	
In this article, we construct bipartite graphs which are cospectral for both the adjacency and normalized Laplacian matrices using  the notion of  partitioned tensor products. This extends the construction of Ji, Gong, and Wang \cite{ji-gong-wang}.  Our proof of the cospectrality of adjacency matrices simplifies the proof of the bipartite case of Godsil and McKay's construction \cite{godsil-mckay-1976}, and shows that the corresponding normalized Laplacian matrices are also cospectral. We partially characterize the isomorphism in Godsil and McKay's construction, and  generalize Ji et al.'s characterization of the isomorphism to biregular bipartite graphs. The essential idea in characterizing  the isomorphism uses Hammack's cancellation law as opposed to Hall's marriage theorem used by Ji et al.
    \end{abstract}


\section{Introduction}

We consider simple and undirected graphs. Let $G=(V(G),E(G))$ be a graph with the vertex set $V(G)=\{1,2,\ldots, n\}$ and the edge set $E(G)$. If two vertices $i$ and $j$ of $G$ are adjacent, we denote it by $i\sim j$. For a graph $G$ on $n$ vertices, the \textit{adjacency matrix} $A(G)=[a_{ij}]$ is the  $n \times n$ matrix defined by
$$a_{ij}=\begin{cases}1,& \mbox{if}~\ i\sim j,\\0,& \mbox{otherwise}.\end{cases}$$ The \textit{spectrum }of a graph $G$ is the set of all eigenvalues of $A(G)$, with corresponding multiplicities. Two graphs are \textit{cospectral for the adjacency matrices} if they have the same adjacency spectrum. It has been a longstanding problem to characterize graphs that are determined by their spectrum \cite{vandam-haemers, devel-vandam-haemers}. If any graph which is cospectral with $G$ is also isomorphic to it, then $G$ is said to be \textit{determined by its spectrum} (DS graph for short), otherwise we say that the graph $G$ has a cospectral mate or we say that $G$ is not determined by its spectrum (NDS for short). To show that a graph is NDS, we provide a construction of a cospectral mate. In \cite{schw}, Schwenk proved that almost all trees are NDS. In \cite{godsil-mckay-1982}, Godsil and McKay provided a method for constructing cospectral nonisomorphic graphs. In \cite{vandam-haemers}, van Dam and Haemers mentioned that: If we were to bet, it would be for: \lq almost all graphs are DS\rq. Later Haemers conjectured the same in \cite{haem-conj}.  For each vertex $i$ of a graph $G$, let $d_i$ denote the \textit{degree} of the vertex $i$. Let $D(G)$ denote the diagonal \textit{ degree matrix} whose $(i,i)$-th entry is $d_i$. Then the matrix $L(G)=D(G)-A(G)$ is  the \textit{Laplacian matrix} of the graph $G$, and if $G$ has no isolated vertices, then the matrix $\mathcal{L}(G) = I - D(G)^{-\frac{1}{2}} A(G) D(G)^{-\frac{1}{2}}$
is  the \textit{normalized Laplacian matrix} of $G$. The {normalized Laplacian specturm} of a graph $G$ is the spectrum of $\mathcal{L}(G)$. Two graphs are \textit{cospectral for the normalized Laplacian matrices} if they have the same normalized Laplacian spectrum. For more details, we refer to \cite{ brou-haem, chung1, godsil-mckay-1982,  vandam-haemers, devel-vandam-haemers}.

In \cite{godsil-mckay-1976}, Godsil and McKay constructed cospectral graphs for the adjacency matrices using the notion of partitioned tensor products of matrices (See Section 2 for the definition). Recently, Ji, Gong, and Wang proposed a construction for cospectral bipartite graphs for the adjacency and normalized Laplacian matrices using the unfolding technique \cite{ji-gong-wang}. This construction is a generalization of the unfoldings of a bipartite graph considered by Butler\cite{but-lama}.  In this paper, first, we note that the proof of the construction of cospectral bipartite graphs \cite[Theorem 2.1]{ji-gong-wang} can be done by expressing the matrices involved as the partitioned tensor products, and our proof works for larger classes of graphs. This is done in Theorem \ref{cospec} (for adjacency matrices) and Theorem \ref{construction} (for normalized Laplacian matrices). Also, the proof of Theorem \ref{cospec} provides an alternate proof of  Godsil and McKay's result for the bipartite graphs.

Weichsel proved that if $G_1$ and $G_2$ are two connected bipartite graphs, then their direct product $G_1 \times G_2$ has exactly two connected bipartite components \cite{weichsel}.   Jha, Klav\v{z}ar and Zmazek \cite{jha-klav-zmaz} showed that if either $G_1$ or $G_2$ admits an automorphism that interchanges its partite sets, then the components of $G_1 \times G_2$ are isomorphic.  Hammack \cite{hammack} proved that the converse is also true.  Hammack's proof uses a cancellation property (see Theorem \ref{cancellation}), which we call Hammack's cancellation law. Surprisingly, we are able to use this result to prove the isomorphism of the cospectral graphs that we construct in Section \ref{construc-sec}.   The characterization theorem for isomorphism of Ji et al.  \cite[Theorem 3.1]{ji-gong-wang},  is a particular case of our result. Also, their proof uses Hall's Marriage theorem, whereas we do not.

The outline of this paper is as follows: In Section \ref{prelims}, we include some needed known results for graphs and matrices. In Section \ref{construc-sec}, the main results about the construction of cospectral bipartite graphs for the adjacency and normalized Laplacian matrices are stated and proved. Section \ref{par-charac-eta} is devoted to the study of the existence of isomorphisms between the cospectral pairs constructed in Section \ref{construc-sec}.

\section{Preliminaries}\label{prelims}

The notion of partitioned tensor products of matrices is used extensively in this article. This is closely related to the well known Kronecker product of matrices. The \emph{Kronecker product} of matrices $A = [a_{ij}]$ of size $m \times n$ and $B$ of size $p \times q$, denoted by $A \otimes B$, is the $mp \times nq$ block matrix $[a_{ij}B]$.
%
%
	The\textit{ partitioned tensor product }of two partitioned matrices $M = \begin{bmatrix}
		U & V \\
		W & X \\
	\end{bmatrix}$ and $H = \begin{bmatrix}
		A & B \\
		C & D \\
	\end{bmatrix}$, denoted by $M  \underline{\otimes} H$, is defined as $ \begin{bmatrix}
		U\otimes A & V\otimes B \\
		W\otimes C & X\otimes D \\
	\end{bmatrix}.$  Given the matrices $U$, $V$, $W$ and $X$, define $\mathcal{I}(U,X)= \begin{bmatrix}U& 0\\0 &X\end{bmatrix}$ and $\mathcal{P}(V,W)= \begin{bmatrix}0 &V\\W& 0\end{bmatrix}$ where $0$ is the zero matrix of appropriate order. A $2\times2$ block matrix is \textit{diagonal} (resp., \textit{an anti diagonal}) block matrix if it is of the form $\mathcal{I}(U,X)$ (resp., $\mathcal{P}(V,W)$).   The above notions were introduced by Godsil and McKay \cite{godsil-mckay-1976}.  The following proposition is easy to verify.

\begin{prop}\label{parti-diag}
	Let $Q$ and $R$ be the matrices of the form $\mathcal{I}(Q_1,Q_2)$ and $\mathcal{I}(R_1,R_2)$, respectively. If $M = \begin{bmatrix}
		U & V \\
		W & X \\
	\end{bmatrix}$ and $H = \begin{bmatrix}
		A & B \\
		C & D \\
	\end{bmatrix}$ are $2 \times 2$ block matrices, then $$(Q \underline{\otimes} R)(M \underline{\otimes} H)=(QM) \underline{\otimes} (RH).$$ The same holds true when the matrices $Q$ and $R$ are both of the form $\mathcal{P}(Q_1,Q_2)$ and $\mathcal{P}(R_1,R_2)$, respectively.
\end{prop}
%

Two matrices $A$ and $B$ are said to be \textit{equivalent}, if there exists invertible matrices $P$ and $Q$ such that $Q^{-1}AP=B$. If the matrices $P$ and $Q$ are orthogonal, then matrices $A$ and $B$ are said to be \textit{orthogonally equivalent}. If the matrices $P$ and $Q$ are permutation matrices, then matrices $A$ and $B$ are said to be \textit{permutationally equivalent}.  Using the singular value decomposition, it is easy to see that any square matrix is orthogonally equivalent to its transpose.


A square matrix $A$ is said to be a \textit{PET matrix }if it is {permutationally equivalent} to its transpose.  If the set of row sums of an $n \times n$ matrix $A$ is different from the set of columns sums of $A$, then $A$ is non-PET.


We recall the cancellation law of matrices given by Hammack.

\begin{thm}\cite[Lemma 3]{hammack}\label{cancellation} Let $A$, $B$ and $C$ be $(0,1)$-matrices. Let $C$ be a non-zero matrix and $A$ be a square matrix with no zero rows. Then, the matrices $C\otimes A$ and $C\otimes B$ are permutationally equivalent if and only if $A$ and $B$ are permutationally equivalent. Similarly, the matrices $A\otimes C$ and $B\otimes C$ are permutationally equivalent if and only if $A$ and $B$ are permutationally equivalent.
\end{thm}
%

An \textit{isomorphism} of two graphs $G_1$ and $G_2$ is a bijection $f: V(G_1) \longrightarrow V(G_2)$ such that any two vertices $u$ and $v$ are adjacent in $G_1$ if and only if $f(u)$ and $f(v)$ are adjacent in $G_2$. Two graphs $G_1$ and $G_2$ are \textit{isomorphic }if there exists an isomorphism between them. It is easy to see that $G_1$ and $G_2$ are isomorphic if and only if the corresponding adjacency matrices are permutationally similar.
An \textit{automorphism} of a graph $G$ is an isomorphism from the graph $G$ to itself and the set of automorphisms $\aut(G)$ of a graph is a group with respect to the composition of functions. Every automorphism of a graph $G$ on $n$ vertices can be represented by an $n \times n$ permutation matrix. Thus $\aut(G)$ can be identified with the set of permutation matrices $P$ such that $P^TA(G)P=A(G)$.

A graph $G$ is \textit{bipartite} if its vertex set can be partitioned into two parts $X$ and $Y$ such that every edge has one end in $X$ and the other end in $Y$. We refer to $V(G)=X\cup Y$ as a bipartition of $G$, and $X$ and $Y$ as the \textit{partite sets} of $G$. A bipartite graph is \textit{balanced} if its partite sets have the same number of elements. If $G$ is a  bipartite graph with the adjacency matrix $\begin{bmatrix}0&B\\B^T&0\end{bmatrix}$, then the  matrix $B$ is the \textit{biadjacency matrix} of $G$.
Let $G_1$ and $G_2$ be two isomorphic bipartite graphs, and let  $V(G_i) = X_i \cup Y_i$ be the bipartition of $G_i$ for $i \in \{1,2\}$. An isomorphism $f$ from $G_1$ to $G_2$ \textit{respects} the partite sets if it satisfies either $f(X_1)=X_2$ and $f(Y_1)=Y_2$ or $f(X_1)=Y_2$ and $f(Y_1)=X_2$.  If $G_1$ and $G_2$ are two connected isomorphic bipartite graphs, then any isomorphism between them respects the partite sets.




Let $G$ be a bipartite graph whose partite sets are $X$ and $Y$. An automorphism $f$ of $G$ \textit{fixes} the partite sets if $f(X)=X$ and $f(Y)=Y$, and \textit{interchanges} the partite sets if $f(X)=Y$ and $f(Y)=X$.  
\begin{defn}\cite{hammack}
A connected bipartite graph has property $\pi$ if it admits an automorphism that interchanges its partite sets.	
\end{defn}

The next  proposition connects PET matrices to automorphisms of bipartite graphs. We skip the proof.

\begin{prop}\label{bip-aut}
	Let $G$ be a  bipartite graph with the adjacency matrix $\begin{bmatrix}0&B\\B^T&0\end{bmatrix}$. Then the biadjacency matrix $B$ is PET if and only if there exists an automorphism $f\in \aut (G)$ that interchanges its partite sets, where the partite sets are induced by the biadjacency matrix $B$.
\end{prop}

%

	The \textit{direct product} or \textit{tensor product} $G_1\times G_2$ of graphs $G_1$ and $G_2$ is the graph with vertex set $V(G_1) \times V(G_2) $(the cartesian product $V(G_1)$ and $V(G_2)$) and two vertices $(g,h)$ and $(g',h')$ are adjacent in $G_1 \times G_2$ if and only if $g$ is adjacent to $g'$ in $G_1$ and $h$ is adjacent to $h'$ in $G_2$. The adjacency matrix of the graph $G_1 \times G_2$ is given by $A(G_1) \otimes A(G_2)$. Here $A(G_i)$ denotes the adjacency matrix of the graph $G_i$ (for $i = 1,2$), and $\otimes$ denotes the Kronecker product.

\section{Construction of the cospectral pairs}\label{construc-sec}

In this section, we  give a construction of cospectral bipartite graphs for both the adjacency and the normalized Laplacian matrices.
Let $I_m$ and $0_n$ denote the identity and zero matrices of orders $m$ and $n$, respectively. Let $V$ be an $m \times n$ matrix and $B$ be a $p \times q$ matrix. Define the matrices $L = \begin{bmatrix}
0       & V  \\
V^T     & 0
\end{bmatrix}$, $H= \begin{bmatrix}
0       & B  \\
B^T     & 0
\end{bmatrix}$ and $H^{\#} = \begin{bmatrix}
0       & B^T  \\
B     & 0
\end{bmatrix}$.   Note that the matrices
$L \underline{\otimes} H=\begin{bmatrix}
	0       & V \otimes B  \\
	V^T \otimes B^T     & 0
\end{bmatrix}$ \text{and } $L  \underline{\otimes} H^{\#}=\begin{bmatrix}
	0       & V \otimes B^T  \\
	V^T \otimes B     & 0
\end{bmatrix}$
are of orders $mp+nq$ and $mq+np$, respectively.

For an $n \times n$ symmetric $(0,1)$ matrix $A$ with zero diagonal entries, let $G_A$ denote the simple graph whose adjacency matrix is $A$. The $(0,1)$ matrices $V$, $B$ and $B^T$ are the biadjacency matrices for the bipartite graphs $G_L$, $G_H$ and $G_{H^{\#}}$, respectively. Since $H$ and $H^{\#}$ are permutationally similar, the graphs $G_{H}$ and $G_{H^{\#}}$ are isomorphic. The next theorem is about  a construction of  cospectral graphs for the adjacency matrices. The proof of this theorem could found in \cite{godsil-mckay-1976}. Our proof gives an alternate simple proof for this result, and the proof idea could be extended for normalized Laplacian matrices as well.

\begin{thm} \label{cospec}
The bipartite graphs $G_{L  \underline{\otimes} H}$ and $G_{L  \underline{\otimes} H^{\#}}$ are cospectral for the adjacency matrices if and only if at least one of the bipartite graphs $G_L$ or $G_H$  is balanced.
\end{thm}

\begin{proof}

First let us show if either $m=n$ or $p=q$, the matrices $L  \underline{\otimes} H$ and $L  \underline{\otimes} H^{\#}$ are orthogonally similar, and hence they are cospectral.
	
	\textbf{Case 1:} Let $m=n$.
	Then $V$ is a square matrix and $V$ is orthogonally equivalent to $V^T$. Thus there exist two orthogonal matrices $R_1$ and $R_2$ such that $R_1^TVR_2=V^T$. Define $R=\mathcal{P}(R_1,R_2)$. Now,
	\begin{gather*}
		R^TLR=\begin{bmatrix}0&R_1\\R_2&0\end{bmatrix}^T\begin{bmatrix}0&V\\V^T&0\end{bmatrix}\begin{bmatrix}0&R_1\\R_2&0\end{bmatrix}
		=\begin{bmatrix}0&R_2^TV^TR_1\\R_1^TVR_2&0\end{bmatrix}\\
		=\begin{bmatrix}0&V\\V^T&0\end{bmatrix}=L.
	\end{gather*}
	Let $Q=\mathcal{P}(I_p,I_q)$ and $P=R \underline{\otimes} Q$. Then $Q$ is a permutation matrix, $P$ is an orthogonal
	matrix and $Q$ satisfies $Q^THQ=H^{\#}$. Now
	\begin{align*}
		P^T(L \underline{\otimes}H)P&=(R \underline{\otimes}Q)^T(L \underline{\otimes}H)(R \underline{\otimes}Q)\\
			&=(R^TLR) \underline{\otimes}(Q^THQ)\\
		&=L \underline{\otimes}H^{\#}.
	\end{align*}
    Note that the second step uses Proposition \ref{parti-diag}. Thus
    $G_{L \underline{\otimes} H}$ and $G_{L \underline{\otimes} H^{\#}}$ are cospectral.
	\textbf{Case 2:} Let $p=q$.
	Then the matrix $B$ is orthogonally equivalent to $B^T$. Hence, there exist two orthogonal matrices $Q_1$ and $Q_2$ such that $Q_1^TBQ_2=B^T$. Let $Q=\mathcal{I}(Q_1,Q_2)$. Then $	 Q^THQ  = H^{\#}.$
	Let $R=\mathcal{I}(I_m,I_n)$ and $P=R \underline{\otimes} Q$. Then $R^TLR = L$. Rest of the proof of this case is similar to that of Case 1.
	
	Conversely, let the matrices $L  \underline{\otimes} H$ and $L  \underline{\otimes} H^{\#}$ have the same spectrum. Then $mp+nq=mq+np$, and hence  $(m-n)(p-q)=0$. Thus the result follows.
\end{proof}

Next, we establish an identity for  the partitioned tensor product for the normalized Laplacian matrices, which is useful in the construction of cospectral graphs for the normalized Laplacian matrices.

\begin{lem}\label{lapla-decom}
	Let $G_1$ and $G_2$ be two bipartite graphs with no isolated vertices. If the matrices $A(G_1)$ (resp., $A(G_2)$) and $\mathcal{L}(G_1)$ (resp., $\mathcal{L}(G_2)$) are partitioned into $2 \times 2$ matrices conformally with the partite sets, then $\mathcal{L}(G_{A(G_1) \underline{\otimes} A(G_2)})=2I - (\mathcal{L}(G_1) \underline{\otimes} \mathcal{L}(G_2))$.
\end{lem}
\begin{proof}
	Let $D({G_1})$ and $D({G_2}) $ denote the degree matrices corresponding to the graphs $G_1$ and $G_2$ respectively.
	 Then,
	\begin{align*}
		\mathcal{L}(G_{A(G_1) \underline{\otimes} A(G_2)})&=I-\bigg(D(G_{A(G_1) \underline{\otimes} A(G_2)})^{-1/2}(A(G_1) \underline{\otimes} A(G_2))D(G_{A(G_1) \underline{\otimes} A(G_2)})^{-1/2}\bigg)\\
		&=I-\bigg(D(G_1)^{-1/2}A(G_1)D(G_1)^{-1/2}\bigg) \underline{\otimes} \bigg(D(G_{2})^{-1/2}A(G_2)D(G_{2})^{-1/2}\bigg)\\
		&=I-\bigg((I - \mathcal{L}(G_1)) \underline{\otimes} (I - \mathcal{L}(G_2))\bigg)\\
		&=2I - (\mathcal{L}(G_1) \underline{\otimes}\mathcal{L}(G_2)).\qedhere
	\end{align*}
\end{proof}

In the next theorem, we establish  the cospectrality for the normalized Laplacian matrices of  partitioned tensor products graphs $G_{L  \underline{\otimes} H}$ and $G_{L  \underline{\otimes} H^{\#}}$.

\begin{thm}\label{construction}
Let $G_L$ and $G_H$ be two bipartite graphs with no isolated vertices. The bipartite graphs $G_{L \underline{\otimes} H}$ and $G_{L \underline{\otimes} H^{\#}}$ are cospectral for the normalized Laplacian matrices if and only if at least one of $G_{L}$ or $G_{H}$  is balanced.
\end{thm}

\begin{proof}
  Let $D(G_L)$, $D(G_H)$ and $D(G_{H^{\#}})$ denote the degree matrices for the graphs $G_{L}$, $G_{H}$ and $G_{H^{\#}}$, respectively. Let either $m=n$ or $p=q$.

\textbf{Case 1:} Suppose $m=n$. Then $V$ is an $n \times n$ matrix. Let $D(G_L)=\mathcal{I}(C_1,C_2)$ where $C_1$ and $C_2$ are $n \times n$ diagonal degree matrices of the respective partite sets. Since $G_{L}$ does not have any isolated vertices, $C_1^{-1/2}$ and $C_2^{-1/2}$ exist. Let $E=C_1^{-1/2}VC_2^{-1/2}$.  Then there exist two orthogonal matrices $R_1$ and $R_2$ such that $E=R_2^TE^TR_1$. Set $R=\mathcal{P}(R_1,R_2)$. Now,
\begin{align*}
\mathcal{L}(G_L)&=I-D(G_L)^{-1/2}A(G_L)D(G_L)^{-1/2}\\
&=I-\begin{bmatrix}
	0&E\\
	E^T&0
\end{bmatrix}\\
&=I-\begin{bmatrix}
	0&R_2^TE^TR_1\\
	R_2ER_1^T&0
\end{bmatrix}\\
&=\begin{bmatrix}
	0&R_1\\R_2&0
\end{bmatrix}^T\left (I-\begin{bmatrix}
0&E\\
E^T&0
\end{bmatrix}\right )\begin{bmatrix}
0&R_1\\R_2&0
\end{bmatrix}\\
&=R^T\mathcal{L}(G_L)R.
\end{align*}
The permutation matrix $Q=\mathcal{P}(I_p,I_q)$ satisfies $Q^T\mathcal{L}(G_H)Q=\mathcal{L}(G_{H^{\#}})$, and  $P=R \underline{\otimes}Q$ is an orthogonal matrix. By Proposition \ref{parti-diag}, it follows that $P^T(\mathcal{L}(G_L) \underline{\otimes} \mathcal{L}(G_H))P=\mathcal{L}(G_L) \underline{\otimes} \mathcal{L}(G_{H^{\#}})$.

\textbf{Case 2:} Suppose $p=q$. Then $B$ is a $p \times p$ matrix.  Let $D(G_H)=\mathcal{I}(D_1,D_2)$ where $D_1$ and $D_2$ are $p \times p$ diagonal matrices. Then, $D(G_{H^{\#}})=\mathcal{I}(D_2,D_1)$. Since $G_{H}$ does not have any isolated vertices, $D_1^{-1/2}$ and $D_2^{-1/2}$ exist. Let $F=D_1^{-1/2}BD_2^{-1/2}$.  Then there exist two orthogonal matrices $Q_1$ and $Q_2$ such that $Q_1^TFQ_2=F^T$.
Let $Q=\mathcal{I}(Q_1,Q_2)$, $R$ be the identity matrix such that $R^T\mathcal{L}(G_L)R=\mathcal{L}(G_L)$ and $P=R \underline{\otimes} Q$. Then $P^T(\mathcal{L}(G_L) \underline{\otimes} \mathcal{L}(G_H))P=\mathcal{L}(G_L) \underline{\otimes} \mathcal{L}(G_{H^{\#}})$.

In both the cases, the matrices $\mathcal{L}(G_L) \underline{\otimes} \mathcal{L}(G_H)$ and $\mathcal{L}(G_L) \underline{\otimes} \mathcal{L}(G_{H^{\#}})$ are orthogonally similar. By Lemma \ref{lapla-decom}, we have $\mathcal{L}(G_{L \underline{\otimes} H})=2I -\mathcal{L}(G_L) \underline{\otimes} \mathcal{L}(G_H)$ and $\mathcal{L}(G_{L \underline{\otimes} H^{\#}})=2I - \mathcal{L}(G_L) \underline{\otimes} \mathcal{L}(G_{H^{\#}})$. Then the matrices $\mathcal{L}(G_{L \underline{\otimes} H})$ and $\mathcal{L}(G_{L \underline{\otimes} H^{\#}})$ are orthogonally similar, and hence they are cospectral.

Converse is easy to verify.
\end{proof}

\begin{rem}
By taking $V = J_{m,n}, J_{1,n}$, and  $J_{1,2}$,  in  Theorem \ref{cospec} and Theorem \ref{construction}, we get the constructions of Ji et al. given in  \cite[Theorem 2.1]{ji-gong-wang}, Kannan and Pragada given in \cite[Theorem 3.1]{kannan-pragada}, and Butler given in \cite[Theorem 2.1]{but-lama}, respectively.
\end{rem}

\section{Property $\eta$ and Isomorphism}\label{par-charac-eta}
We assume from here on that the bipartite graphs $G_L$ and $G_H$ have no isolated vertices, and they are not necessarily connected. In this section, we investigate the existence of isomorphism between the cospectral pair obtained in Theorem \ref{cospec} and Theorem \ref{construction}. We show that, under appropriate restrictions, the isomorphism is closely related to the PET matrices and the automorphisms of bipartite graphs.

If a bipartite graph $G$ is disconnected, then $G $ has more than one bipartition. Next we extend the property $\pi $ for disconnected bipartite graphs. Let $G$ be a bipartite graph with biadjacency matrix $B$.  We say that  $G$ has property $\pi$ with respect to $B$, if $B$ is  a PET matrix. We define the canonical partite sets of the partitioned tensor product graph $G_{L \underline{\otimes} H}$ (resp., $G_{L \underline{\otimes} H^\#}$) to be the bipartition induced by the biadjacency matrix $V \otimes B$ (resp., $V \otimes B^T$), where $V$ and $B$ are biadjacency matrices of $G_L$ and $G_H$, respectively.


The next theorem connects the property $\pi$  with  the isomorphism between the graphs $G_{L \underline{\otimes} H}$ and $G_{L \underline{\otimes} H^{\#}}$.

\begin{thm}\label{inter-imp-iso}
Let  $G_{L}$ and $G_{H}$ be bipartite graphs (not necessarily connected) with biadjacency matrices $V$ and $B$, respectively. Then the following statements are equivalent:
\begin{itemize}
	\item[(a)] There exists an isomorphism between the graphs $G_{L \underline{\otimes} H}$ and $G_{L \underline{\otimes} H^{\#}}$ which respects the partite sets of the canonical bipartitions.
	\item[(b)]   At least one of $G_{L}$ or $G_{H}$ has property $\pi$.
\end{itemize}
\end{thm}

\begin{proof}
		\textbf{(a)$\implies$(b)} By  assumption, there exists a permutation matrix $P$  of the form either $\mathcal{I}(P_1,P_4)$ or $\mathcal{P}(P_2,P_3)$, where $P_i$ is a permutation matrix for $i \in \{1,2,3,4\}$, such that
		$$P^T\begin{bmatrix}0&V\otimes B\\V^T\otimes B^T&0\end{bmatrix}P=\begin{bmatrix}0&V\otimes B^T\\V^T\otimes B&0\end{bmatrix}.$$
		
		If $P=\mathcal{I}(P_1,P_4)$,  then $P_1^T(V\otimes B)P_4=V\otimes B^T$. The matrices $V\otimes B$ and $V\otimes B^T$ are permutationally equivalent. Then, by  Theorem \ref{cancellation}, $B$ is PET.
		If $P=\mathcal{P}(P_2,P_3)$, then $P_3^T(V^T\otimes B^T)P_2=V\otimes B^T$.  Hence  $V$ is PET.
		
	\textbf{(b)$\implies$(a)}
Let either $G_L$ or $G_H$ has property $\pi$.
Suppose $G_L$ has the property $\pi$, that is, $G_{L}$ admits an automorphism that interchanges its partite sets. Then, there exist two permutation matrices $R_2$ and $R_3$ such that $R^TLR=L$ where $R=\mathcal{P}(R_2,R_3)$. Set $Q=\mathcal{P}(I_p,I_q)$ and $P=R \underline{\otimes}Q$. Then $Q^THQ=H^{\#}$, and
	\begin{align*}
		P^T(L \underline{\otimes}H)P&=(R \underline{\otimes}Q)^T(L \underline{\otimes}H)(R \underline{\otimes}Q)\\
		&=(R^TLR) \underline{\otimes}(Q^THQ)\\
		&=L \underline{\otimes}H^{\#}.
	\end{align*}
	Thus the graphs $G_{L \underline{\otimes} H}$ and $G_{L \underline{\otimes} H^{\#}}$ are isomorphic, and the isomorphism induced by the permutation matrix $P$ respects the partite sets.
	
Proof idea of the other case is similar.
\end{proof}
Weichsel proved that if $G_L$ and $G_H$ are two connected bipartite graphs, then $G_L\times G_H$ has exactly two connected bipartite components \cite{weichsel}.
Godsil and McKay noted that the disjoint union of the bipartite graphs $G_{L\underline{\otimes} H}$ and $G_{L\underline{\otimes} H^{\#}}$ is the direct product $G_{L}\times G_{H}$ \cite{godsil-mckay-1976}. Hence, if $G_L$ and $G_H$ are connected, then the two connected components of $G_{L\otimes H}$ are $G_{L\underline{\otimes} H}$ and $G_{L\underline{\otimes} H^{\#}}$.
Using these observations,  we obtain Hammack's result as a corollary of Theorem \ref{inter-imp-iso}.

\begin{cor}\cite[Theorem 1]{hammack}
	Suppose $G_1$ and $G_2$ are connected bipartite graphs. The two components of $G_1 \times G_2$ are isomorphic if and only if at least one of $G_1$ or $G_2$ has the property $\pi.$
\end{cor}

The bipartite graphs  constructed by Ji, Gong and Wang have the
property that any isomorphism between them respects the partite sets \cite[Lemma 3.2]{ji-gong-wang}.
From Theorem \ref{inter-imp-iso}, we observe that respecting partite sets  is  the key property for characterizing isomorphism in terms of PET matrices (property $\pi$). To this end, we define  property $\eta$, which relaxes the connectedness assumption in the Hammack's result and includes a broader class of bipartite graphs.

\begin{defn}\label{prop-eta-def}
Two bipartite graphs $G_{L}$ and $G_{H}$ are said to have \textit{property $\eta$}, if whenever the bipartite graphs $G_{L \underline{\otimes} H}$ and $G_{L \underline{\otimes} H^{\#}}$ are isomorphic, there exists an isomorphism between $G_{L \underline{\otimes} H}$ and $G_{L \underline{\otimes} H^{\#}}$ that respects the partite sets of the canonical bipartitions.
\end{defn}

Next we show that any pair of connected bipartite graphs have  property $\eta$ and thus property $\eta$ is relaxation of connectedness requirement from Hammack's result \cite[Theorem 1]{hammack}.

\begin{thm}\label{connec-eta-cond}
	Let $G_L$ and $G_H$ be connected bipartite graphs. Then, they have property $\eta$.
\end{thm}
\begin{proof}
	Suppose $G_L$ and $G_H$ are connected, and $G_{L\underline{\otimes} H}$ and $G_{L\underline{\otimes} H^{\#}}$ are isomorphic. Since $G_{L\underline{\otimes} H}$ and $G_{L\underline{\otimes} H^{\#}}$ are connected bipartite graphs, thus any isomorphism between them respects the partite sets. Hence $G_{L}$ and $G_{H}$ have property $\eta$.
\end{proof}

A \textit{biregular bipartite graph} is a bipartite graph $G$ for which any two vertices in the same partite sets have the same degree as each other. If degree of the vertices in one of the partite sets is $k$ and  degree of the vertices in the other partite set is $l$, then the graph is said to be $(k,l)$-biregular. We say that a biregular bipartite graph has distinct degrees if $k \neq l$.  Next we prove that  property $\eta$ is satisfied if  one of the bipartite graphs is biregular with distinct degrees.

%

\begin{thm}\label{biregu-eta}
	Let $G_{L}$ and $G_{H}$ be bipartite graphs. If $G_{L}$ is a non-empty $(k,l)$-regular bipartite graph with  $k\neq l$, then $G_L$ and $G_H$ have property $\eta$.
\end{thm}

\begin{proof}
Let the graphs $\Gamma_1=G_{L\underline{\otimes} H}$ and $\Gamma_2=G_{L\underline{\otimes} H^{\#}}$ be isomorphic. Then,
$\Gamma_1$ and $\Gamma_2$ are cospectral. By Theorem \ref{cospec}, the graph $G_H$ is balanced, since $G_L$ cannot be balanced as $k\neq l$.

Let $V(\Gamma_i) = X_i \cup Y_i$ be the canonical vertex partitions of the graphs $\Gamma_i$ for $i = {1,2}$. Without loss of generality, assume that $k <l$. Let $f$ be an isomorphism from $\Gamma_1$ to $\Gamma_2$. Let $b_i$ and $b_i'$ denote the $i^{th}$ row sum of the matrices $B$ and $B^T$, respectively. Let $x_1$ be the vertex of maximum degree in $X_1$.  Suppose that $f(x_1)\in Y_2$. Then $d_{\Gamma_1}(x_1)=lb_i$ for some $1\leq i\leq p$,  and $d_{\Gamma_2}(f(x_1))=kb_j$ for some $1\leq j\leq p$. Since the isomorphism preserves the degrees, we have $lb_i=kb_j$. Since $x_1$ has maximum degree in $X_1$, $b_i\geq b_j$ for any $1\leq j\leq p$, and hence $kb_j\geq lb_j$. If $b_j \neq 0$, then $k\geq l$,  a contradiction to the initial assumption that $k< l$. Hence, if $x_1\in X_1$, then $f(x_1)\in X_2$. If $b_j=0$ then $lb_i=kb_j$, $b_i=0$. But $x_1$ is a vertex of maximum degree $lb_i$ in the set $X_1$ and thus $B=0$. So we could choose $f(x_1)\in X_2$.
In any case, $f(x_1)\in X_2$.

Let $x_1,\ldots,x_{rm}$ be the vertices of $X_1$ with the same maximum degree such that $d_{\Gamma_1}(x_{1+(s-1)m})=\ldots=d_{\Gamma_1}(x_{m+(s-
	1)m})=lb_{i_s}$ for $s\in \{1,2,\ldots, r\}$ where $b_{i_1}=\ldots=b_{i_r}$ for $1\leq i_1, \ldots, i_r \leq p$.
Then, using the previous argument, $f(x_1),\ldots, f(x_{rm}) \in X_2$ such that $d_{\Gamma_2}(f(x_{1+(s-1)m}))=\ldots=d_{\Gamma_2}(f(x_{m+(s-1)m}))=lb'_{j_s}$ for $s\in \{1,2,\ldots, r\}$ where $b_{j_1}'=\ldots=b_{j_r}'$ for $1\leq j_1,\ldots, j_r\leq p$. Define $B{'}$ to be the matrix obtained by removing the $i_s^{th}$ row and $j_s^{th}$ column of $B$ for all $s\in \{1,2,\ldots, r\}$.
Define $\Gamma_1'$ and $\Gamma_2'$ to be the induced bipartite graphs corresponding to the biadjacency matrices $V\otimes B{'}$ and $V\otimes B{'}^T$,
respectively. Since $\Gamma_1'$ and $\Gamma_2'$ are isomorphic as well, apply the same argument for $\Gamma_1'$ and $\Gamma_2'$ until all the rows and columns of $B$ are exhausted. Thus $f(X_1)=X_2$ and hence $f(Y_1)=Y_2$.

Similarly, if  $k>l$, then consider the set of vertices of maximum degree in $Y_1$ and show that $f(Y_1)=Y_2$ and hence $f(X_1)=X_2$. Hence, $G_{L}$ and $G_{H}$ satisfy property $\eta$.
\end{proof}

Note that at each step, the vertices from both the partite sets of the induced bipartite graphs of $\Gamma_1$ and $\Gamma_2$ are being removed. This is justified since our motive is to first show just $f(X_1)=X_2$. Ji, Gong and Wang in Lemma 3.2. \cite{ji-gong-wang} remove vertices from only $X_1$ and $X_2$ at each step.

Since we are interested in the construction of cospectral nonisomorphic graphs, we use this result to construct  cospectral graphs that are not isomorphic.

\begin{thm}\label{wang-genera}
	Let $G_{L}$ and $G_{H}$ be bipartite graphs. Let $G_{L}$ be a biregular bipartite graph with distinct degrees and let $G_H$ be balanced. 
	Then the graphs $G_{L\underline{\otimes} H}$ and $G_{L\underline{\otimes} H^{\#}}$ are  nonisomorphic if and only if $G_{H}$ does not admit an automorphism that interchanges its partite sets.
\end{thm}
\begin{proof}
	Since $G_L$ is a biregular bipartite graph with distinct degrees, the corresponding $m\times n$ biadjacency matrix $V$ has constant row sum $k$ and constant column sum $l$. Since the sum of row sums must be the same as the sum of column sums, we have $km=ln$. But $k\neq l$, hence $m\neq n$. Hence, $G_L$ has unequal partition sizes. Since $G_H$ has equal partition sizes,  by Theorem \ref{cospec} and Theorem \ref{construction},  the graphs $G_{L\underline{\otimes} H}$ and $G_{L\underline{\otimes} H^{\#}}$ are cospectral. Now, as $G_L$ has unequal partitions sizes, it doesn't admit an automorphism that interchanges its partite sets. Hence, the condition for non-isomorphism follows from Theorems \ref{inter-imp-iso} and \ref{biregu-eta}.
\end{proof}

	Now as a corollary, we obtain the result of Ji, Gong and Wang.
	\begin{cor}\label{wang-res}\cite[Theorem 3.1]{ji-gong-wang}
		Let $V=J_{m,n}$ such that $m\neq n$ and let $B$ is a square matrix. Then, the bipartite graphs $G_{L\underline{\otimes} H}$ and $G_{L\underline{\otimes} H^{\#}}$ are cospectral for the adjacency as well as the normalized Laplacian matrices, and they are isomorphic if and only if $B$ is PET.
	\end{cor}
	\begin{proof}
		Since $V=J_{m,n}$ and $m\neq n$, the corresponding bipartite graph $G_L$ is a biregular bipartite graph with distinct degrees. Hence, the result follows from Theorem \ref{wang-genera}.
	\end{proof}
%
	
Now, let us illustrate the construction given in Theorem \ref{wang-genera} with an example.
\begin{ex}
	The following pair of  matrices $V$ and $B$ satisfy all the conditions stated in  Theorem \ref{biregu-eta} and corresponding graphs are illustrated below.
	\begin{align*}V =\begin{bmatrix}1&0\\1&0\\0&1\\0&1\end{bmatrix} \text{and } \; B =\begin{bmatrix}1&1&0\\1&0&1\\1&0&0\end{bmatrix}.
	\end{align*}
	
	In the following figure, $G_{L}$ and $G_{H}$ denote the bipartite graphs with biadjacency matrices $V$ and $B$, respectively.  The graphs $G_{L\underline{\otimes} H}$ and $G_{L \underline{\otimes} H^{\#}}$ are the cospectral nonisomorphic pairs. Note that this example of cospectral bipartite graphs is not obtainable
	from the results of Ji et al. and Hammack.
%
%

	\begin{figure}[H]
	\centering
	\includegraphics[scale=0.4]{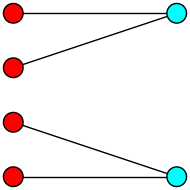}
	\includegraphics[scale=0.4]{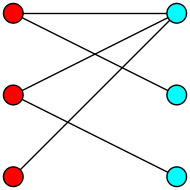}
	\includegraphics[scale=0.5]{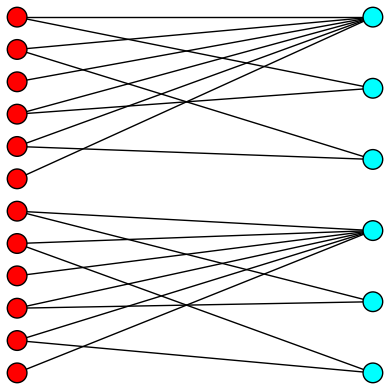}
	\includegraphics[scale=0.5]{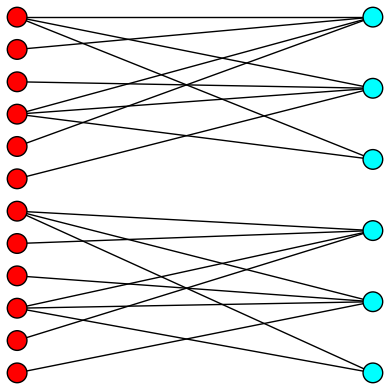} \caption{$G_{L}$, $G_{H}$, $G_{L\underline{\otimes} H}$ and $G_{L\underline{\otimes} H^{\#}}$}\label{fig2}
\end{figure}
\end{ex}

\subsubsection*{Acknowledgment:}
We are indebted to the referees for the comments and suggestions. M. Rajesh Kannan would like to thank the Department of Science and Technology, India, for financial support through the projects MATRICS (MTR/2018/000986).

\bibliographystyle{plain}
\nocite{hitesh}
\bibliography{references}
\end{document}